\documentclass[a4paper]{siamltex}

\usepackage{a4wide}
\usepackage[english]{babel}
\usepackage{amssymb,amsfonts,amsmath,latexsym}
\usepackage{graphicx}
\usepackage{url}

\newcommand{\bm}[1]{\mathbf{#1}}

\newcommand{\bell}{\boldsymbol{\ell}}
\newcommand{\R}{{\mathbb R}}

\newcommand{\cI}{{\mathcal{I}}}

\newcommand{\e}{{\mathrm{e}}}

\newtheorem{remark}{Remark}[section]

\numberwithin{equation}{section}

\begin{document}

\title{Identifying the lights position in photometric stereo \\
under unknown lighting}

\author{A.Concas\thanks{Department of Mathematics and Computer Science,
University of Cagliari, viale Merello 92, 09123 Cagliari, Italy. E-mail:
\texttt{anna.concas@unica.it}, \texttt{kate.fenu@unica.it},
\texttt{rodriguez@unica.it}.}
\and
R. Dess\`i\thanks{Departiment of Electrical and Electronic Engineering,
University of Cagliari, Piazza D'Armi, 09123 Cagliari, Italy. E-mail:
\texttt{richidessi@gmail.com}, \texttt{vanzi@diee.unica.it}.}
\and
C. Fenu\footnotemark[1]
\and
G. Rodriguez\footnotemark[1]
\and
M. Vanzi\footnotemark[2]}

\maketitle

\begin{abstract}
Reconstructing the 3D shape of an object from a set of images is a classical
problem in Computer Vision.
Photometric stereo is one of the possible approaches.
It stands on the assumption that the object is observed from a fixed point of
view under different lighting conditions.
The traditional approach requires that the position of the light sources is
accurately known.
It has been proved that the lights position can be estimated directly from the
data when at least 6 images of the observed object are available.
In this paper, we present a Matlab implementation of the algorithm for solving
the photometric stereo problem under unknown lighting, and propose a simple
shooting technique to solve the bas-relief ambiguity.
\end{abstract}

\begin{keywords}
photometric stereo, low-rank approximation, finite differences, shape from
shading, computer vision.
\end{keywords}

\section{Introduction}\label{sec:intro}

\emph{Shape from shading} is a typical problem in Computer
Vision~\cite{horn1989,klette1998,slabaugh2001}.
It exploits shading information for recovering the 3D shape of an object from a
from a set of discretized 2D projections (e.g., digital pictures), when the
object is illuminated by a single light source.
If only one picture is available, the \emph{shape from shading} problem is not
well posed, as its solution is not unique.
A way to solve this problem and remove the uncertainty is to use more than one
image. 
There are two main possible approaches to the problem.
Stereo vision, sometimes generalized to multiple views vision or
\emph{multiview}~\cite{dyer2001}, assumes the availability of different
observations of an object obtained by varying the point of view, but not the illumination.
The pictures are typically obtained by a set of fixed cameras, or extracted
from the frames of a movie shot by a moving camera.
Photometric stereo (PS) \cite{woodham1980} is a \emph{shape from shading}
method that uses a fixed camera and a movable light to acquire a set of images
that embed shape and color (albedo) information of the framed object
\cite{christensen1994}.
The ideal PS requires the lights position and intensity to be accurately known
\cite{barsky2003}, a deviation from this requirement often results in a
distorted reconstruction.

Various attempts have been made to estimate the lights position directly from
the data; see, e.g., \cite{basri2007,chen2006} where a linear combination of
special functions (spherical harmonics) is employed.
The problem was solved in \cite{hayakawa1994}, where it was shown that at
least 6 images under different lighting conditions are needed.
Such a result releases the constraint on the precise positioning
of the light sources, making the acquisition process much simpler.

We devoted various papers to the application of photometric stereo to rock art
documentation, in particular to the 3D reconstruction of the decorations found 
in the ``Domus de Janas'', a particular kind of Neolithic tombs typical of
Sardinia, Italy \cite{dmrtv15,mrtv15,mrv16,vmdrt14}. 
Some B.S. students were also involved in this study 
\cite{dessi14,floris17,mannu14,stocchino15}.
In the archaeology setting, a 3D restoration technique based on easy data
acquisition is crucial, because the findings are frequently located in
uncomfortable positions and the current protocols exclude any physical contact
for creating replicas. 
The difficulty to access specific sites, often associated to a large number of
items to be documented, makes it impractical to use other 3D reconstruction
techniques, like laser scanning, characterized by long acquisition time and
large instrumentation cost.
Indeed, cheap instrumentation would allow for parallel operation on different
findings by a team of researchers.

In this paper we present a Matlab implementation for the solution of the
photometric stereo problem under unknown lighting.
First, the normal vector field is determined by a technique due to
Hayakawa~\cite{hayakawa1994}, which can estimate the lights position directly
from the data when at least 6 images of the observed object are available.
Then, an explicit representation of the observed surface, based on an
orthographic projection, is obtained by integrating the normal field, 
This is done by solving a Poisson equation with either Dirichlet or Neumann
boundary conditions.
Finally, we propose a practical shooting procedure to circumvent the well known
bas-relief ambiguity \cite{basrelief99}, and demonstrate the performance of the
developed software on both synthetic and experimental data sets.

Section~\ref{sec:notation} introduces the notation adopted in the paper, while
Section~\ref{sec:known} resumes the solution of the problem when the lights
position is known in advance.
Photometric stereo under unknown lighting is studied in
Section~\ref{sec:unknown}, where we recall a theorem which gives sufficient
conditions for the existence of the solution.
Section~\ref{sec:orient} illustrates a procedure for determining the right
orientation of the surface.
A selection of numerical results are illustrated in Section~\ref{sec:examples},
and Section~\ref{sec:future} describes possible future developments.

\section{Notation}\label{sec:notation}

Let us consider an object placed at the origin of a reference system in $\R^3$.
The object is observed from a fixed camera, the $z-$axis coincides with the
optical axis and it is directed from the object to the camera.
The point of view is assumed to be at infinite distance from the observed
object (ortographic projection), and different pictures of the object are
taken, each one with a different light direction.
Each image has resolution $(r+2)\times(s+2)$, and the real length of the
horizontal side of each image is $A$.
Assuming the pixels to be square, we let the length of the vertical side be
$B=(s+1)h$, with $h=A/(r+1)$.
Each picture defines a domain $\Omega=[-A/2,A/2]\times[-B/2,B/2]$, and induces
the discretization
\begin{equation}
\begin{aligned}
x_i &= -A/2 + ih, \quad i =0,\ldots,r+1,\\  
y_j &= -B/2 + jh, \quad j=0,\ldots,s+1.
\end{aligned}
\label{mesh}
\end{equation}


We let the surface of the object be represented by $z=u(x,y)$, with
$(x,y)\in\Omega$. Then,
\begin{equation}\label{gradnorm}
\nabla u(x,y) = \begin{bmatrix} 
\displaystyle \frac{\partial u(x,y)\strut}{\partial x\strut} \\ 
\displaystyle \frac{\partial u(x,y)\strut}{\partial y\strut}
\end{bmatrix}
= \begin{bmatrix} u_x\strut \\ u_y\strut \end{bmatrix}, \quad
\bm{n}(x,y) = \frac{(-u_x,-u_y,1)^T}{\sqrt{1+\|\nabla{u}\|^2}},
\end{equation}
denote the gradient of $u$ and the normal vector to the surface of the object,
respectively.

As it is customary, images are vectorized by ordering the pixels
lexicographically.
The pixel of coordinates $(i,j)$ takes the index $k=(i-1)s+j$, where
$k=1,\ldots,p$, and $p$ is the number of pixels in the image.
We will set either $p=(r+2)(s+2)$ or $p=rs$ depending on whether the boundary
pixels are considered in the discretization or not.
For each point in the discretization of $\Omega$, we write indifferently
$$
\begin{aligned}
u(x_i,y_j) &= u_{i,j} = u_k, \\
u_x(x_i,y_j) &= (u_x)_{i,j} = (u_x)_k, \\
u_y(x_i,y_j) &= (u_y)_{i,j} = (u_y)_k, \\
\bm{n}(x_i,y_j) &= \bm{n}_{i,j} = \bm{n}_k,
\end{aligned}
$$
using the two-index notation to refer to the values on the grid, and the
one-index notation to identify the values after the vectorization.

We assume that $q$ pictures are available, each one with light source at infinite
distance from the origin along the direction 
$$
\bell_t = \begin{pmatrix} \ell_{1t} \\ \ell_{2t} \\ \ell_{3t} \end{pmatrix},
\quad t=1,\ldots,q.
$$
Each vector $\bell_t$ stems from the object to the light source and its
Euclidean norm is proportional to the light intensity. 
This introduces an undetermined proportionality constant in the problem.
The image vectors are denoted by $\bm{m}_1,\bm{m}_2,\ldots,\bm{m}_q\in\R^p$.
The aim is to reconstruct the 3D shape of the object.

\section{Photometric stereo with known lighting}\label{sec:known}

If the surface of the object is \emph{Lambertian} \cite{klette1998}, then the
light intensity observed at each point is proportional to the angle between the
normal to the observed surface at that point and the light direction.
This model is usually referred to as Lambert's cosine law.
It can be stated as
\begin{equation}\label{lambertlaw}
\rho(x,y) \, \langle \bm{n}(x,y), \bell_t \rangle = \cI_t(x,y), \quad
t=1,\ldots,q
\end{equation}
where the scalar function $\rho(x,y)$ represents the \emph{albedo} at each 
surface point and it  keeps into account the partial light absorption of
that portion of the surface.
The light intensity at each point of the $t$th image is denoted by
$\cI_t(x,y)$, and $\langle\cdot,\cdot\rangle$ is the usual inner
product in $\R^3$.
When the albedo is constant, the object is said to be a \emph{Lambertian
reflector}.

Summing up, the classical assumptions for this model are the following:
\begin{itemize}
\item the surface is Lambertian;
\item the light sources are placed at infinite distance;
\item no portion of the surface is shaded in all the pictures;
\item the camera is sufficiently far from the object so that perspective
deformations can be neglected.
\end{itemize} 
In this section, we assume that the light directions $\bell_t$,
$t=1,\ldots,q$, are known.

The continuous formulation \eqref{lambertlaw} of Lambert's law leads to a
Hamilton--Jacobi differential model; see \cite{mecca2011,mecca2013} for a
thorough study. Here we briefly review its construction.

Setting
$$
\bell_t = \begin{pmatrix} \widetilde{\bell}_t \\ \ell_{3t} \end{pmatrix}, \quad
\widetilde{\bell}_t\in\R^2,
$$
equation \eqref{lambertlaw} becomes
$$
\rho(x,y) \, \frac{\langle-\nabla{u(x,y)},\widetilde{\bell}_t\rangle+\ell_{3t}}
	{\sqrt{1+\|\nabla{u(x,y)}\|^2}} = \cI_t(x,y), \quad t=1,\ldots,q,
$$
or, equivalently,
$$
\cI_t(x,y) \sqrt{1+\|\nabla{u(x,y)}\|^2} +
\rho(x,y) \bigl( \langle \nabla{u(x,y)},\widetilde{\bell}_t\rangle-\ell_{3t}
\bigr) = 0.
$$
By imposing Dirichlet boundary conditions, one obtains the system of $q$ first
order nonlinear PDEs of Hamilton--Jacobi type
$$
\begin{cases}
H_t(x,y,\nabla{u(x,y)})=0, \quad & t=1,\ldots,q, \\
u(x,y)=g(x,y), & (x,y) \in \partial\Omega.
\end{cases}
$$

Following \cite{mecca2013}, we obtain for $t=1$
$$
\sqrt{1+\|\nabla{u(x,y)}\|^2} = \rho(x,y) 
\frac{\langle-\nabla{u(x,y)},\widetilde{\bell}_1\rangle+\ell_{31}}{\cI_1(x,y)}
$$
and we substitute this expression in the equations corresponding to
$t=2,\ldots,q$, to obtain
\begin{equation*}
\left(\langle-\nabla{u(x,y)},\widetilde{\bell}_1\rangle-\ell_{31}\right) 
\cI_t(x,y)
= \left(\langle-\nabla{u(x,y)},\widetilde{\bell}_t\rangle-\ell_{3t}\right) 
\cI_1(x,y).
\end{equation*}

This shows that the minimal number of images for the problem to be well-posed
is 2. 
Nevertheless, if $q=2$ the solution may not exist for particular light
orientations.
Considering $q>2$ leads to a least-squares approach, that may be effective 
to reduce the influence of noise in experimental data sets, without making data
acquisition significantly harder.
In any case, knowing accurately the lights position $\bell_t$ is a strong
requirement.

After $u(x,y)$ is computed, the albedo is given by
$$
\rho(x,y) = \frac{\cI_t(x,y)}{\langle \bm{n}(x,y),\bell_t \rangle}, 
\quad \text{for any $t=1,\ldots,q$}.
$$
Conditions for the existence of solutions are discussed in \cite{kozera1991},
and in \cite{mecca2013} the problem is studied under more realistic
assumptions; see also \cite{stocchino15}.

%

The main disadvantage of the Hamilton--Jacobi model is that the operator to be
inverted depends upon the data. The matrix of the linear system obtained
through the discretization may be singular or severely ill-conditioned in
certain lighting conditions. However, this problem can be tackled by suitably
choosing the position of the light sources.

An alternative approach for the 3D reconstruction of a Lambertian surface
is based on splitting the computation into two subproblems, one concerning the
computation of the normal vectors and another in the reconstruction of the
surface by the solution of a Poisson equation.

The first step consists of immediately discretizing Lambert's law on a
regular grid, and determining the normal vector field to the surface by solving
a matrix equation. This operation is mostly influenced by the lights position,
Then, the divergence of the normal field is numerically approximated, to obtain
a discretization of the Laplace operator, and the 3D profile of the observed
object is recovered by solving a Poisson partial differential equation.
This part is uniquely concerned with the camera position.

The advantages of this approach are that the computation is decoupled
into simpler problems, and that it allows for the solution of the unknown
lighting case, as it will be shown in the next section.
A drawback is that this procedure requires a larger number of images than the
Hamilton--Jacobi formulation, i.e., at least 3.
This is not a substantial problem in applications, since usually dozens of
images can be easily made available.

Let us apply discretization \eqref{mesh} to equation \eqref{lambertlaw}.
We ignore for the moment the boundary pixels, where we will impose suitable
boundary conditions, and rearrange the internal pixels by the lexicographical
ordering.
Denoting by $\rho_k$ and $\bm{n}_k$, respectively, the value of the albedo and
the (normalized) normal vector to the surface at the $k$th pixel, then the
following relation holds at each point of each picture
\begin{equation}
\rho_k \bm{n}_k^T \bell_t = m_{kt}, \quad k=1,\ldots,p,\quad t=1,\ldots,q.
\label{eq:lambert}
\end{equation}
The scalars $m_{kt}$ represent the radiation $\cI_t(x,y)$ reflected by the small
area near the $k$th pixel when illuminated from the direction $\bell_t$, that
is, the components of the vectors $\bm{m}_1,\bm{m}_2,\ldots,\bm{m}_q\in\R^p$,
containing the vectorized images.

By defining the matrices
\[
\begin{aligned}
D &= \diag(\rho_1,\ldots,\rho_p)\in\R^{p\times p}, \\
L &= [\bell_1,\ldots,\bell_q]\in\R^{3\times q}, \\
N &= [\bm{n}_1,\ldots,\bm{n}_p]\in\R^{3\times p}, \\
M &= [\bm{m}_1,\ldots,\bm{m}_q]\in\R^{p\times q},
\end{aligned}
\]
the equations~\eqref{eq:lambert} can be grouped into the matrix equation
\begin{equation}
D N^T L = M.
\label{eq:fact}
\end{equation}
When the lights positions are known, we first compute
\begin{equation}
\widetilde{N}^T = M L^\dagger,
\label{eq:lsprob}
\end{equation}
where $L^\dagger$ is the Moore-Penrose pseudoinverse of $L$ \cite{bjo96}.
Then, the matrices $D$ and $N$, defining the albedo and the normal vectors, can
be computed from the factorization $ND=\widetilde{N}$ by simply normalizing
the columns of $\widetilde{N}$.

For the solution of \eqref{eq:lsprob} to be unique, it is necessary that 
$q\geq 3$, from which we see that the minimum number of images required to
obtain the normal field by this approach is 3.

Once the field of the normal vectors to the surface is obtained, we consider
the vectors
\[
((u_x)_k,(u_y)_k,-1)^T = -\frac{\bm{n}_k}{(\bm{n}_k)_3},
\]
obtained by normalizing to -1 the third component of the normals $\bm{n}_k$;
see \eqref{gradnorm}.
We numerically differentiate the first two components of the above vectors to
obtain an approximation on the grid \eqref{mesh} of the Laplacian
$f(x,y)=u_{xx}+u_{yy}$.
To do that, we employ the following formula, based on the second order centered
finite differences approximation for the first derivative
\begin{equation}\label{lapapprox}
f_{i,j} = f(x_i,y_j) \approx \frac{(u_x)_{i+1,j}-(u_x)_{i-1,j}}{2h} + 
\frac{(u_y)_{i,j+1}-(u_y)_{i,j-1}}{2h}.
\end{equation}

Then, the 3D profile of the object, represented by the explicit function
$z=u(x,y)$, can be recovered by solving the Poisson partial differential
equation
\begin{equation}
\Delta u(x,y) = f(x,y),
\label{poisson}
\end{equation}
where $\Delta$ denotes the Laplace operator, and suitable boundary condition
must be imposed to ensure unicity of solution.

We discretize the Poisson equation by a second order finite differences scheme
and initially consider Dirichlet boundary conditions.
Consider the equation~\eqref{poisson} on the rectangle
$[-A/2,A/2]\times[-B/2,B/2]$, with boundary conditions
$$
\begin{aligned}
u(x,-B/2) & =  \phi_1(x),\quad u(x,B/2) = \phi_2(x), \quad x\in[-A/2, A/2], \\
u(-A/2,y) & =  \psi_1(y),\quad u(A/2,y) = \psi_2(y), \quad y\in[-B/2, B/2].
\end{aligned}
$$
Let $u(x_i, y_j)= u_{i,j}$ and $f(x_i, y_j)= f_{i,j}$ at each point of the
mesh~\eqref{mesh}. Then, the boundary values are denoted by
$u_{i,0}=\phi_1(x_i)$, $u_{i,s+1}=\phi_2(x_i)$, $u_{0,j}=\psi_1(y_j)$, and
$u_{r+1,j}=\psi_2(y_j)$.

Discretizing the Poisson equation by the well-known five point scheme with
stepsize $h$, we obtain the linear system
\begin{equation}
u_{i-1,j}+u_{i,j-1}-4u_{i,j}+u_{i,j+1}+u_{i+1,j} = \tilde{f}_{i,j},
\label{discrPois}
\end{equation}
for $i=1,\dots,r$ and $j=1,\dots,s$, with $\tilde{f}_{i,j}=h^2f_{i,j}$.
We remark that approximating $f_{i,j}$ by \eqref{lapapprox} does not
deteriorate the quality of the results, as both \eqref{lapapprox} and the five
point scheme produce an approximation of order $O(h^2)$.
This assertion has been verified numerically.

By aggregating the mesh points $u_{i,j}$ by columns,
we obtain the following pentadiagonal system of size $p=rs$
\begin{equation}\label{pentdiagsyst}
\begin{cases}
\begin{aligned}
&T\bm{u}_1+I_s\bm{u}_2 &&= \bm{b}_1, \\
&I_s\bm{u}_{i-1}+T \bm{u}_i +I_s\bm{u}_{i+1} &&= \bm{b}_i,
	\quad && i = 2,\dots,r-1, \\
&I_s\bm{u}_{r-1}+T \bm{u}_{r}&&= \bm{b}_{r},
\end{aligned}
\end{cases}
\end{equation}
where $I_s$ denotes the identity matrix of size $s$,
$$
T = \begin{bmatrix}
-4 & 1 & & \\
1 & -4 & \ddots & \\
 & \ddots & \ddots & 1\\
 & & 1 & -4
\end{bmatrix}\in \R^{s\times s},\quad 
\bm{u}_i = \begin{bmatrix}
u_{i,1}\\
u_{i,2}\\
\vdots \\
u_{i,s}
\end{bmatrix}\in \R^{s}, \quad i=1,\ldots,r.
$$
The right-hand side vectors are
$$
\begin{aligned}
\bm{b}_1 &= \begin{bmatrix}
\tilde f_{1,1}\\
\tilde f_{1,2}\\
\vdots \\
\tilde f_{1,s}\\
\end{bmatrix}
-\begin{bmatrix}
u_{0,1}+u_{1,0} \\
u_{0,2} \\
\vdots\\
u_{0,s}+u_{1,s+1} 
\end{bmatrix}, \\
\bm{b}_r &= \begin{bmatrix}
\tilde f_{r,1} \\
\tilde f_{r,2} \\
\vdots \\
\tilde f_{r,s} \\
\end{bmatrix}
-\begin{bmatrix}
u_{r+1,1}+u_{r,0} \\
u_{r+1,2} \\
\vdots\\
u_{r+1,s}+u_{r,s+1} 
\end{bmatrix},
\end{aligned}
$$
and
$$
\bm{b}_i = \begin{bmatrix}
\tilde{f}_{i,1}\\
\tilde{f}_{i,2}\\
\vdots \\
\vdots \\
\tilde{f}_{i,s}\\
\end{bmatrix}-\begin{bmatrix}
u_{i,0} \\
0 \\
\vdots \\
0 \\
u_{i,s+1} \\
\end{bmatrix}\in \R^{s}, \quad i=2,\ldots,r-1.
$$

The system has the condensed representation $A\bm{u} = \bm{b}$, where 
$$\bm{u} = \begin{bmatrix}
\bm{u}_{1}\\
\bm{u}_{2}\\
\vdots\\
\vdots\\
\bm{u}_{r}\\
\end{bmatrix}\in \R^{p}, \quad \bm{b} = \begin{bmatrix}
\bm{b}_{1}\\
\bm{b}_{2}\\
\vdots\\
\bm{b}_{r-1}\\
\bm{b}_{r}\\
\end{bmatrix}\in \R^{p},
$$
and
$$
A = \begin{bmatrix}
T & I_{s} & & \\
I_{s} & T & \ddots & \\
 & \ddots & \ddots & I_{s}\\
 & & I_{s} & T 
\end{bmatrix}\in \R^{p\times p}.
$$
This square linear system of size $p$ (number of pixels) can be
solved by any general direct or preconditioned iterative method suited for
large sparse matrices \cite{gvl96}, or, specifically, by a fast Poisson solver
\cite{golub1970,chan1987}.

In practical photometric stereo, one usually focuses on the case of homogeneous
Dirichlet boundary conditions, i.e.,
$\phi_1(x)=\phi_2(x)=\psi_1(y)=\psi_2(y)=0$, which corresponds to assuming that
the observed object stands on a flat background.

When the value of the function $u(x,y)$ on the boundary is unknown, there may
be information on the normal derivatives of the solution.
This amounts to imposing the following Neumann boundary conditions on the
horizontal and vertical boundaries of the domain 
\begin{align}
-\frac{\partial u(x,-B/2)}{\partial y} &= \mu_1(x), & 
\frac{\partial u(x,B/2)}{\partial y} &= \mu_2(x),
\label{neum1} \\
-\frac{\partial u(-A/2,y)}{\partial x} &= \nu_1(y), &  
\frac{\partial u(A/2,y)}{\partial x} &= \nu_2(y).
\label{neum2}
\end{align}
Neumann boundary conditions are not sufficient to make the problem well posed.
In fact, the solution is determined up to an additive constant, so the slope of
a point of the solution has to be fixed arbitrarily.

In this case, the solution at the boundary is to be determined too, and the
number of unknowns $u_{i,j}$ increases from $rs$ to $p=(r+2)(s+2)$.
Equation \eqref{discrPois} is still valid for all the internal points of the
grid, that is, for $i=1,\dots,r$ and $j=1,\ldots,s$, but 
it has to be coupled to the discretization of the conditions \eqref{neum1} and
\eqref{neum2}.
To do this, we employed a one-sided second order discretization, obtaining
on the horizontal boundaries of the domain
$$
\begin{aligned}
3u_{i,0}-4u_{i,1}+u_{i,2} = \tilde{\mu}_{1,i}, \\
u_{i,s-1}-4u_{i,s}+3u_{i,s+1} = \tilde{\mu}_{2,i},
\end{aligned}
$$
with $\tilde{\mu}_{1,i}=2h\mu_1(x_i)$ and $\tilde{\mu}_{2,i}=2h\mu_2(x_i)$, for
$i=1,\dots,r$.
Similarly, on the vertical boundaries, we get
$$
\begin{aligned}
3u_{0,j}-4u_{1,j}+u_{2,j} = \tilde{\nu}_{1,j}, \\
u_{r-1,j}-4u_{r,j}+3u_{r+1,j} = \tilde{\nu}_{2,j},
\end{aligned}
$$
with $\tilde{\nu}_{1,j} = 2h\nu_1(y_j)$ and $\tilde{\nu}_{2,j}=2h\nu_2(y_j)$,
for $j=1,\dots,s$.
On the four corner points, we perform a linear interpolation from the three
neighbour nodes of the grid.
The equation for the corner with coordinates $(x_0,y_0)$ is
$$
u_{0,0} - u_{0,1} - u_{1,0} + u_{1,1} = 0.
$$
Similar equations correspond to the other three corners of the rectangular
domain.

The resulting linear system
\begin{equation}\label{linsys}
B\bm{u} = \bm{c}
\end{equation}
is defined by
\[
B = \begin{bmatrix}
P & Q & S \\
S & R & S \\
& \ddots & \ddots & \ddots \\
& & S & R & S \\
& & S & Q & P
\end{bmatrix}, \quad
\bm{u} = \begin{bmatrix}
\bm{u}_0 \\ \bm{u}_1 \\ \vdots \\ \bm{u}_r \\ \bm{u}_{r+1}
\end{bmatrix}, \quad
\bm{c} = \begin{bmatrix}
\bm{c}_0 \\ \bm{c}_1 \\ \vdots \\ \bm{c}_r \\ \bm{c}_{r+1}
\end{bmatrix},
\]
with
\[
\begin{aligned}
P &= \begin{bmatrix}
1 & -1 \\
& 3 \\
& & \ddots \\
& & & 3 \\
& & & -1 & 1
\end{bmatrix}, \\
Q &= \begin{bmatrix}
-1 & 1 \\
& -4 \\
& & \ddots \\
& & & -4 \\
& & & 1 & -1
\end{bmatrix}, \\
R &= \begin{bmatrix}
3 & -4 & 1 \\
1 & -4 & 1 & \\
& \ddots & \ddots & \ddots & \\
& & 1 & -4 & 1 & \\
& & 1 & -4 & 3
\end{bmatrix},
\end{aligned}
\]
and $S=\diag(0,1,\ldots,1,0)$.

The expression of the right-hand side is the following:
\[
\bm{c}_0 = \begin{bmatrix}
0 \\ \tilde{\nu}_{1,1} \\ \vdots \\ \tilde{\nu}_{1,s} \\ 0
\end{bmatrix}, \quad
\bm{c}_{r+1} = \begin{bmatrix}
0 \\ \tilde{\nu}_{2,1} \\ \vdots \\ \tilde{\nu}_{2,s} \\ 0
\end{bmatrix}, \quad
\bm{c}_i = \begin{bmatrix}
\tilde{\mu}_{1,i} \\ \tilde{f}_{i,1} \\ \vdots \\ \tilde{f}_{i,s} \\ 
\tilde{\mu}_{2,i}
\end{bmatrix}, \quad
i=1,\ldots,r.
\]
\begin{remark}\label{rem:gamma}\rm
To make both the matrix $B$ nonsingular and the solution unique, we substitute
in \eqref{linsys} the equation associated to a chosen internal point of the
domain, say, $(x_i,y_j)$, with the equation $u_{i,j}=\gamma$, where $\gamma$ is
the slope assigned to that point.
\end{remark}

\section{Photometric stereo under unknown lighting}\label{sec:unknown}

The need for an accurate localization of the light sources position 
is a strong limitation for the practical application of photometric stereo.
For this reason, the automatic determination of the light directions has been
of interest to many researchers.

Referring to a \textit{4-sources photometric stereo}, some papers conjecture
that the problem with unknown lighting can be uniquely solved using only 4 images.
In \cite{chen2006} the authors suggest an approach based on  the use of
low-order spherical harmonics for Lambertian objects, while \cite{basri2007}
proposes a method based on the decomposition of the light intensity into a 
linear combination of spherical harmonics.

The problem was actually solved by Hayakawa in \cite{hayakawa1994}.
We briefly review here his results, as they contain the main steps of the
numerical procedure.
The photometric stereo technique under unknown lighting consists of
computing the rank-3 factorization 
\begin{equation}
\widetilde{N}^T L = M,
\label{eq:factilde}
\end{equation}
where $\widetilde{N}=ND$ (see \eqref{eq:fact}), without knowing in advance the
lights location, i.e., the matrix $L$.
This problems has not a unique solution.
Nevertheless, there are some physical constraints which allow one to find a
meaningful solution.

\begin{lemma}\label{lemmaQ}
The matrices $D$, $N$, and $L$, containing the albedo, the normals to the
observed surface, and the lights directions, are determined up to a unitary
transformation, that is, ~\eqref{eq:factilde} is satisfied by the matrix pair
$(Q\widetilde{N},QL)$, for any orthogonal matrix $Q\in\R^{3\times 3}$.
\end{lemma}

\begin{proof}
Any matrix pair $(A^{-T}\widetilde{N},AL)$, with $A\in\R^{3\times 3}$
nonsingular, satisfies~\eqref{eq:factilde}.
Since the normal vectors $\bm{n}_k$ are normalized, the norm of the $k$th
column of $\widetilde{N}$ equals the albedo $\rho_k$, while $\|\bell_t\|$ is
proportional to the light intensity.
This implies that the transformation matrix $A$ has to be orthogonal.
\end{proof}

The above Lemma suggests that the original orientation of the observed object
cannot be determined without further a priori information.
This fact, known as \emph{bas-relief ambiguity} \cite{basrelief99}, should be
expected, since only the relative position between the object and the camera
can be deduced from a set of images.
This indetermination imposes some care in the shape reconstruction, because
there is the possibility of axes reflections in the computation of the
solution, which would alter significantly the shape of the reconstructed object.
We will treat this particular problem in the next section.

In what follows, it is not restrictive to assume $\|\bell_t\|=1$,
$t=1,\ldots,q$.
Indeed, we already noticed that there is an undetermined proportionality
constant in the problem, depending upon the unit of measure adopted for light
intensity, and in the typical experimental setting the pictures are taken using
a flashlight at a fixed distance from the object, which produces a constant
light intensity across the observations.
In particular situations the light intensity may vary, for example when the
size of the observed object requires the use of the sun as a light source,
taking pictures at different times of the day. In this case a light meter can
be used to obtain an estimate of $\|\bell_t\|$.

Let the ``compact'' singular value decomposition (SVD) \cite{gvl96} of the
observations matrix be
\begin{equation}\label{svd}
M = U\Sigma V^T,
\end{equation}
where $\Sigma=\diag(\sigma_1,\ldots,\sigma_q)$ is the diagonal matrix
containing the singular values and $U\in\R^{p\times q}$,
$V\in\R^{q\times q}$ are matrices whose orthonormal columns $\bm{u}_i$ and
$\bm{v}_i$ are the left and right singular vectors, respectively.
In our application $q\ll p$, since the number of pixels in an image is usually
very large, while we would like to obtain a reconstruction using a set of
observations as small as possible.
As we observed in the previous section, it is only required that $q\geq 3$.

When $q$ is small, the SVD factorization can be computed efficiently
by standard numerical libraries; we used the \texttt{svd} function of Matlab
\cite{matlab} even for a quite large value of $p$.
In particular situations, in order to reduce the computation time, one may
compute a partial singular value decomposition by an iterative method; see,
e.g., \cite{BR05,BR13}.

Since the images may be acquired in non-ideal conditions and may be affected by
noise, factorization \eqref{svd} usually has numerical rank $r>3$.
Then, a \emph{truncated SVD} must be performed by setting
$\sigma_4=\cdots=\sigma_q=0$.
We let $W=[\sigma_1\bm{u}_1,\sigma_2\bm{u}_2,\sigma_3\bm{u}_3]^T$
and $Z=[\bm{v}_1,\bm{v}_2,\bm{v}_3]^T$, so that $W^TZ\simeq M$.
This choice produces the best rank-3 approximation to the data matrix $M$ with
respect to both the Euclidean and the Frobenius norm \cite{bjo96}.
The constructive proof of the following theorem from \cite{hayakawa1994} shows
how to obtain the sought matrices $\widetilde{N}$ and $L$ from this initial
factorization.

\begin{theorem}
The normal vectors to the observed surface and the lights position can be
uniquely determined from~\eqref{eq:fact}, up to a unitary transformation, only
if at least 6 images taken in different lighting conditions are available.
\end{theorem}

\begin{proof}
Let us consider the initial rank-3 factorization $W^TZ=M$ described above, with
$W=[\bm{w}_1,\ldots,\bm{w}_p]$ and $Z=[\bm{z}_1,\ldots,\bm{z}_q]$.
Given the assumption on the norms of the vectors $\bell_t$,
we first determine a matrix $B$ such that $\|B\bm{z}_t\|=1$ for each
$t=1,\ldots,q$. This implies solving the system of equations
\begin{equation}
\diag(Z^T G Z) = \bm{1},
\label{eq:diag}
\end{equation}
where $\bm{1}=(1,\ldots,1)^T\in\R^q$ and $G=B^TB$ is a symmetric positive
definite $3\times 3$ matrix.
The matrix $G$ depends upon 6 independent parameters, say, its elements
$g_{ij}$ with $i\leq j$. 
As each equation in~\eqref{eq:diag} is of the type
\[
\bm{z}_t^T G \bm{z}_t = \sum_{i,j=1}^3 z_{it} z_{jt} g_{ij} = 1,
\]
the system~\eqref{eq:diag} can be rewritten in the form
\[
H\bm{g}=\bm{1},
\]
where $\bm{g}=(g_{11},g_{22},g_{33},g_{12},g_{13},g_{23})^T$ and $H$ is a
$q\times 6$ matrix, whose rows are
\[
\begin{bmatrix}
z_{1t}^2 & z_{2t}^2 & z_{3t}^2 & 2z_{1t}z_{2t} & 2z_{1t}z_{3t} & 2z_{2t}z_{3t}
\end{bmatrix},
\quad t=1,\ldots,q.
\]
A necessary condition for the solution vector $\bm{g}$ to be unique is that
$q\geq 6$.
This completes the proof.
\end{proof}
\smallskip

\begin{remark}\label{rem:lights}\rm
The above theorem shows that at least 6 images are needed to reconstruct a
shape by photometric stereo under unknown lighting.
This is only a necessary condition for the unique solvability of the problem,
as $H$ may be rank-deficient even for $q\geq 6$.
In fact, the requirement on the rank of $H$ poses some constraints on the
lights disposition.
For example, a very common experimental approach is to place the light sources
roughly on a circle around the camera, i.e., at a fixed distance $\delta$ from
the origin, on a plane parallel to the observation plane.
This is equivalent to fixing angles $\theta_1,\ldots,\theta_q\in[0,2\pi)$ and
setting
\[
\bell_t = \frac{(\cos\theta_t, \sin\theta_t, \delta)^T}{\sqrt{1+\delta^2}},
\quad t=1,\ldots,q.
\]
This lights placement is not acceptable, because in this case the third
column of the matrix $H$ would be a linear combination of the first two.
So at least one light source should violate this scheme.
Placing the light sources at random positions around the object is often
a safe and easy way to ensure that $H$ is full-rank.
\end{remark}
\smallskip

As the sought matrix $B$ is determined up to a unitary transformation, we
represent it by its QR factorization $B=QR$. 
The ``essential'' factor $R$ can be obtained by the Cholesky factorization 
$G=R^TR$ \cite{gvl96}, while $Q$ cannot be uniquely determined; see
Lemma~\ref{lemmaQ}.
We will discuss a reasonable choice for the matrix $Q$ in the following
section.

Once $Q$ is determined, problem \eqref{eq:factilde} is solved by setting
\begin{equation}\label{lightsrot}
\widetilde{N}=QR^{-T}W \quad \text{and} \quad L=QRZ.
\end{equation}
By normalizing the columns of $\widetilde{N}$ one obtains the diagonal albedo
matrix $D$ and the matrix $N$, whose columns are the normal vectors, such that
$ND=\widetilde{N}$.
The integration of the normal vector field is then performed by one of the
approaches described in Section~\ref{sec:known}.

\section{Determining the right surface orientation}\label{sec:orient}

As we already observed, the matrices $\widetilde{N}$ and $L$ can be determined
only up to a unitary transformation $Q$.
It is nevertheless important to suitably choose $Q$, at least for two reasons.

First of all, the indetermination in the factorization \eqref{eq:factilde} may
introduce axes reflections in the reference system centered at the object, with
the result of capsizing the direction of a part of the normal vectors.

Secondly, the integration procedure in Section~\ref{sec:known} assumes that
the function describing the shape of the object is single-valued and explicit,
that is, has the form $z=u(x,y)$.
The matrix $Q$ should introduce a rotation of the reference system which meets
this assumption.

As observed in \cite{basrelief99}, it is impossible to obtain such information
from the data, without additional a priori information.
In our case, this information comes from the knowledge of the light directions
and from a a good practice in taking the pictures.
The shooting procedure we propose, consists of taking the pictures in a
particular order.
We assume, conventionally, that for the first picture the light is placed at
the right hand of the camera, but this is not restrictive.
What is important, is that the light source is moved counterclockwise around
the object, in the half space containing the camera, and that the pictures are
ordered according to this sequence.
There is no need that the light sources are regularly spaced; on this regard,
see Remark~\ref{rem:lights}.

After determining the matrices $\widehat{N}=R^{-T}W$ and $\widehat{L}=R^{-T}Z$
by the procedure described in Section~\ref{sec:unknown}, we consider the three
columns $\hat{\bell}_t$ of $\widehat{L}$, with
$t=1,\lfloor\frac{q}{3}\rfloor,\lfloor\frac{2q}{3}\rfloor$, where $\lfloor
x\rfloor$ denotes the integer part of $x$.
Given the above shooting procedure, this vector triplet must have a
right-handed (or positive) orientation. This can be checked by the sign of the
determinant of the matrix formed by the vector triplet. If the determinant is
negative, then the factorization procedure introduced an axis reflection. To
restore the original orientation we change the sign of the third row of
$\widehat{N}$ and $\widehat{L}$, which corresponds to inverting the direction
of the $z$ axis.

Once the normal field is well-oriented, we turn to determining a rotation
matrix $Q$ that restores the original orientation, i.e., with the camera
aligned on the $z$ axis.
To approximately identify the direction of the camera as seen from the observed
object, given the proposed shooting procedure, we set
\[
\bm{v}_3 = \sum_{t=1}^q \widehat{\bell}_t.
\]
and assume the vector $\bm{v}_3$ to be the direction of the $z$ axis.
The direction $\bm{v}_1$ of the $x$ axis is arbitrarily obtained by projecting
$\widehat{\bell}_1$ on the plane orthogonal to $\bm{v}_3$, and the $y$ axis by
computing the cross product $\bm{v}_2=\bm{v}_3\wedge\bm{v}_1$.
After normalizing the three vectors, the orthogonal matrix
\[
Q=\begin{bmatrix} \bm{v}_1 & \bm{v}_2 & \bm{v}_3 \end{bmatrix}
\]
determines the sought rotation, which is used for computing the matrices
$\widetilde{N}$ and $L$ in \eqref{lightsrot}.

\section{Numerical experiments}\label{sec:examples}

In this section we illustrate the performance of the algorithm for the solution
of the photometric stereo problem with unknown lighting discussed in the paper.
The Matlab software we developed is available at the web page
\url{http://bugs.unica.it/cana/software/ps3d}.
The first two data sets used in the numerical experiments, as well as the
reconstructed object of \figurename~\ref{synthsurf}, are available on the same web
page as \texttt{mat} files (Matlab data files), and a \texttt{ply} 3D model
file, respectively.
All the computation were performed using Matlab 9.5 on a Debian GNU/Linux
system.
The 3D meshes displayed in the paper were generated by our software and
visualized by the MeshLab open source system for displaying and editing 3D
triangular meshes (\url{www.meshlab.net}).

The software was tested both with synthetic and experimental data sets, in
order to investigate its performance not only when the assumptions on which the
algorithm is based are met, but also in a real-world setting, where the
assumptions are only approximately verified.

\subsection{Synthetic data set}

To assess the accuracy attainable by our implementation in the
ideal situation when all the assumptions of the method are exactly satisfied,
we resorted to a synthetic dataset.
We fixed a disposition of $q=7$ light sources, placing them around the object
at angles 
$(0,\frac{\pi}{4},\frac{3\pi}{4},\pi,\frac{5\pi}{4},\frac{3\pi}{2},\frac{7\pi}{4})$,
and generated a set of digital images by applying the direct model
\eqref{lambertlaw} to the surface represented by the function 
\begin{equation}\label{synthmod}
u(x,y) = \frac{1}{2} \e^x \sin(\pi x) \sin(\pi y),
\end{equation}
on the square domain $[-1,1]\times [-1,1]$. 
Each image is $101\times 101$ pixels, and the albedo equals $\frac{1}{2}$ for
$x^2+y^2<\frac{1}{4}$, and 1 otherwise.
\figurename~\ref{synthdata} and~\ref{synthdataset} show the synthetic surface and
the corresponding data set, respectively.

\begin{figure}[hbt]
\begin{center}
\includegraphics[width=.8\columnwidth]{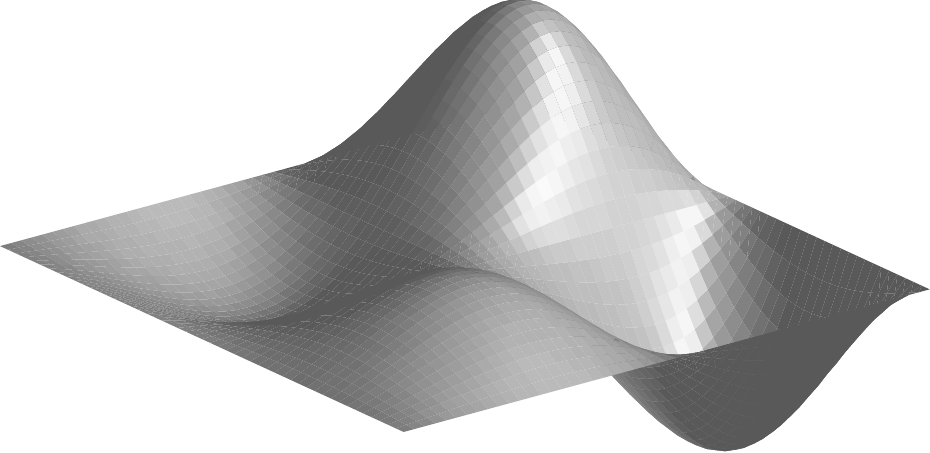}
\caption{A synthetic surface used to investigate the performance of the
algorithm.}
\label{synthdata}
\end{center}
\end{figure}

\begin{figure}[hbt]
\begin{center}
\includegraphics[width=.8\columnwidth]{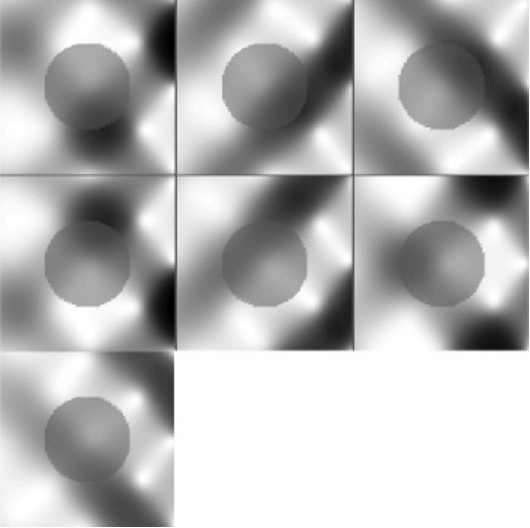}
\caption{Data set collected from the synthetic surface displayed in
\figurename~\ref{synthdata}. It is composed by 7 pictures corresponding to 
different lighting conditions.}
\label{synthdataset}
\end{center}
\end{figure}

\begin{figure}[hbt]
\begin{center}
\hfill
\includegraphics[width=.48\columnwidth]{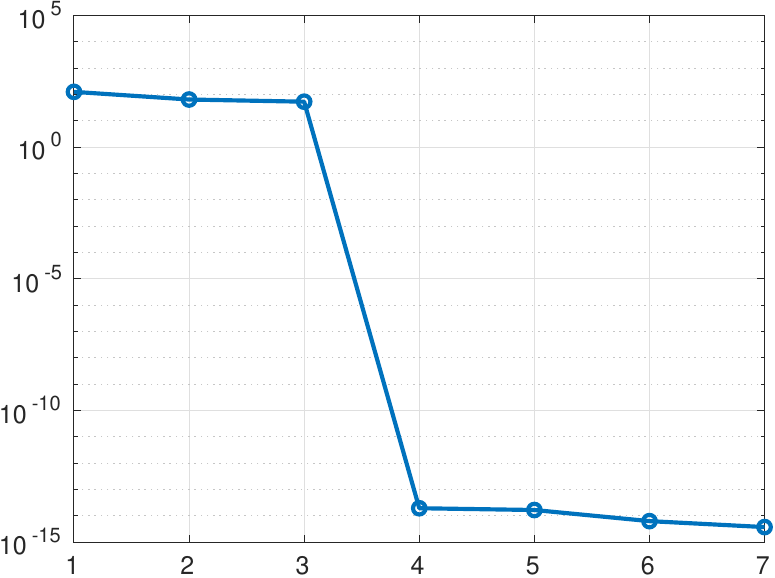}
\includegraphics[width=.48\columnwidth]{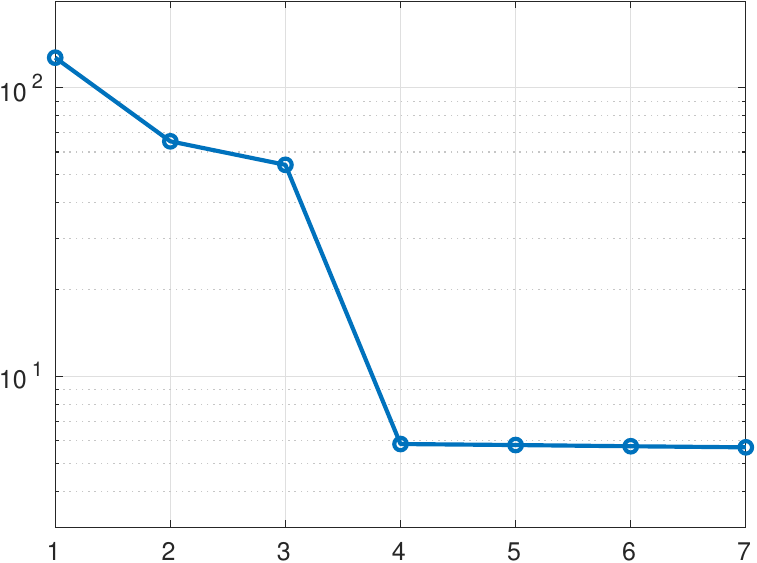}
\caption{On the left, singular values of the data matrix M for the synthetic
data set; on the right, singular values of the same matrix after 10\% Gaussian
noise is added.}
\label{synthsvd}
\end{center}
\end{figure}

The graph on the left of \figurename~\ref{synthsvd} displays the singular values of
the data matrix $M$ from \eqref{eq:fact}, showing that it clearly has numerical
rank $3$.
\figurename~\ref{synthlights} shows the reconstruction of the light vectors and
\figurename~\ref{synthsurf} displays the restored surface.
The light vectors are recovered up to machine precision,
while the relative accuracy on the surface is $2.6\cdot 10^{-4}$, in accord with
the quite large step size $h=\frac{1}{50}$.
The two errors are defined by
$$
E_{\text{lights}} = \frac{\|L-\tilde{L}\|_F}{\|L\|_F}, \quad
E_{\text{surface}} = \frac{\|U-\tilde{U}\|_F}{\|U\|_F},
$$
where $\|\cdot\|_F$ is the Frobenius norm, $(L,U)$ denote the exact matrices
containing the light vectors and the surface slopes, respectively, and
$(\tilde{L},\tilde{U})$ the reconstructed matrices.

\begin{figure}[hbt]
\begin{center}
\includegraphics[width=.8\columnwidth]{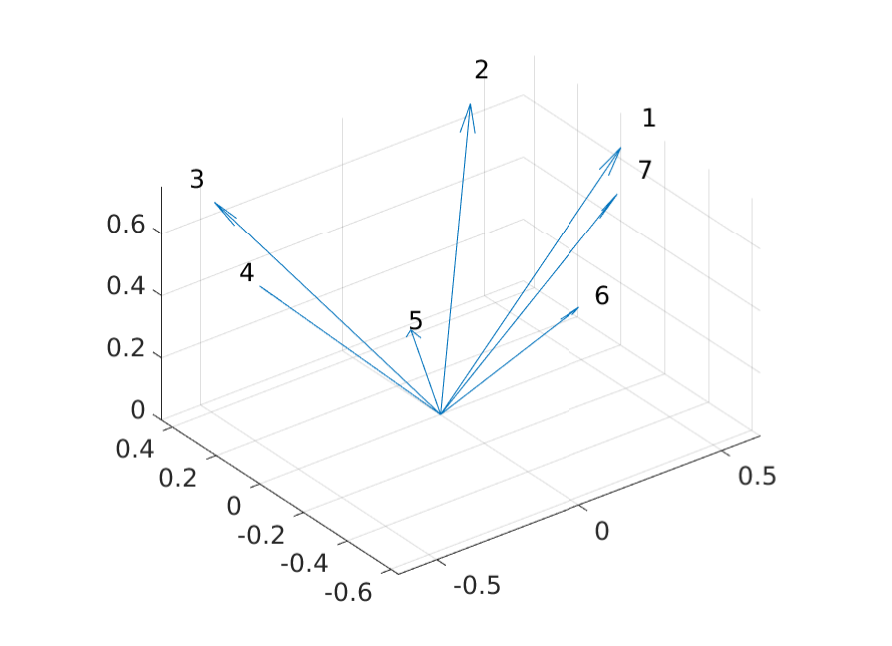}
\caption{Reconstruction of the light directions.}
\label{synthlights}
\end{center}
\end{figure}

\begin{figure}[hbt]
\begin{center}
\includegraphics[width=.8\columnwidth]{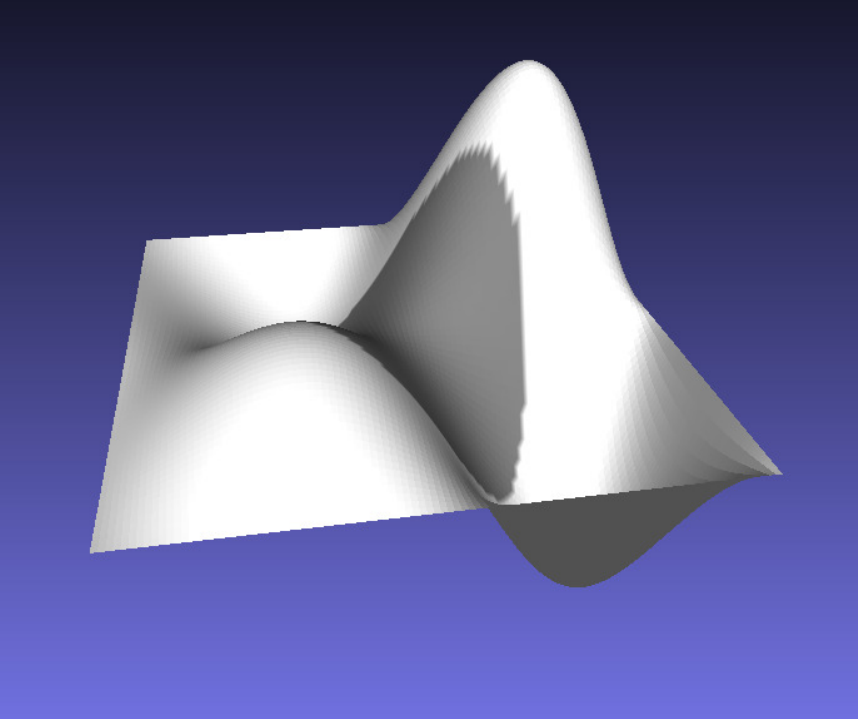}
\caption{Restored surface.}
\label{synthsurf}
\end{center}
\end{figure}

We repeated the same test after introducing 10\% Gaussian noise in the
right-hand side $M$ of \eqref{eq:fact}.
The presence of the noise is reflected in the singular values of $M$, depicted
in the graph on the right of \figurename~\ref{synthsvd}.
Though its rank is 7, it is evident that $M$ can be well approximated by a rank
3 matrix.
The computation is quite steady, as in this case 
$E_{\text{lights}}=3.6\cdot 10^{-3}$, while 
$E_{\text{surface}}=1.5\cdot 10^{-2}$.

\begin{figure}[hbt]
\begin{center}
\includegraphics[width=.8\columnwidth]{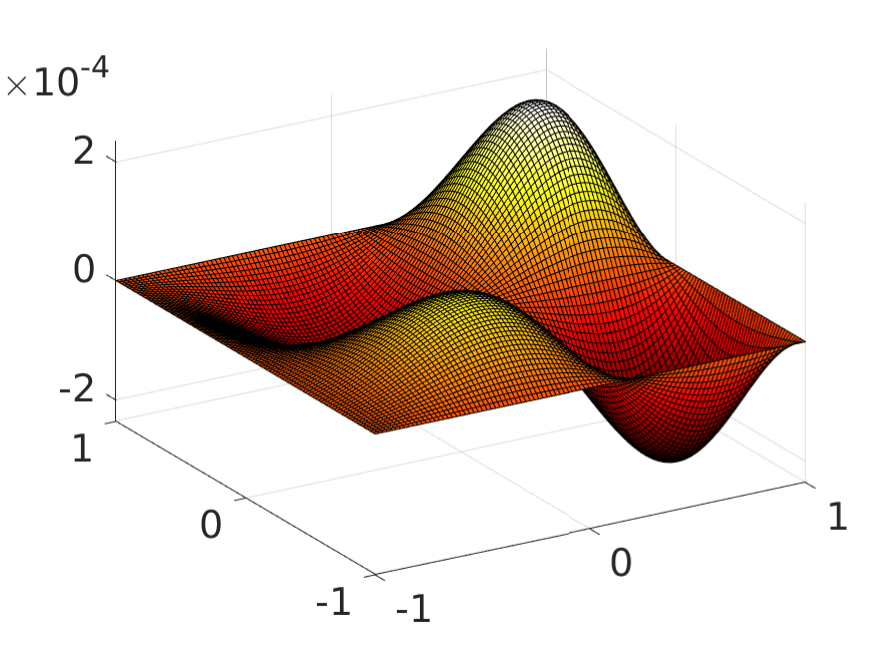}
\caption{Reconstruction error with Dirichlet boundary conditions.}
\label{syntherrdir}
\end{center}
\end{figure}

\begin{figure}[hbt]
\begin{center}
\includegraphics[width=.8\columnwidth]{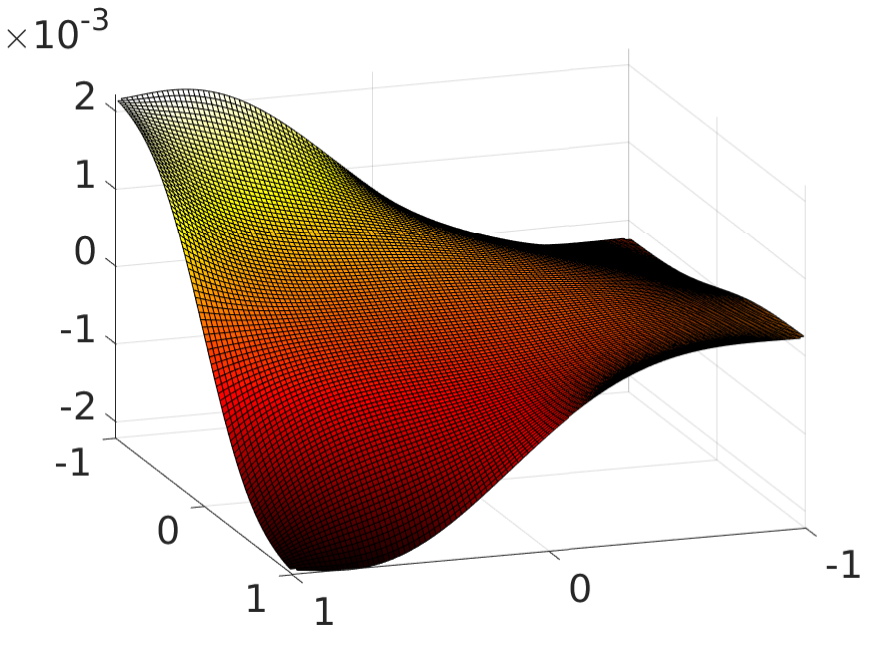}
\caption{Reconstruction error with Neumann boundary conditions.}
\label{syntherrneu}
\end{center}
\end{figure}

The reconstruction in \figurename~\ref{synthsurf} was obtained by imposing
homogeneous Dirichlet boundary conditions, which are exactly verified by model
function \eqref{synthmod}. The error $u(x,y)-\tilde{u}(x,y)$ between the model
and the reconstruction is displayed in \figurename~\ref{syntherrdir}.
The graph in \figurename~\ref{syntherrneu} shows the error corresponding to
exact Neumann conditions. In this case we fixed the value of the solution at
the central point of the domain, that is, $\gamma=u(0,0)=0$; see
Remark~\ref{rem:gamma}.
Neumann conditions produced a less accurate reconstruction in proximity of the
border of the domain, compared to Dirichlet conditions.

To investigate how deviation from ideal lighting influences numerical results,
we repeated the above experiment, with Dirichlet conditions and without noise,
positioning the light sources at finite distance $\kappa A$ from the object,
where $\kappa$ is a scale factor and $A$ is the horizontal width of the
observed scene; in this particular example $A=2$. The synthetic images were
generated by a model based on Lambert's law, which considers incident light
rays, rather than parallel rays. The light directions were recovered by the
procedure described in Section~\ref{sec:unknown}.

When $\kappa<\infty$ the matrix $M$ in \eqref{eq:lsprob} is full-rank, and this
deteriorates the approximation accuracy of the rank-3 factorization constructed
in Section~\ref{sec:unknown}.
We measure the closeness of the matrix $M$ to being rank-3 by the ratio
$\sigma_3/\sigma_4$ between the third and the fourth singular value of $M$.
Table~\ref{tab:rank} reports this ratio together with the errors
$E_{\text{lights}}$ and $E_{\text{surface}}$, for values of $\kappa$ ranging
from $\infty$ to 1. For example, when $\kappa=10$ the distance of the light
source from the origin is 10 times the width of the observed scene.

\begin{table}\label{tab:rank}
\caption{Influence of the distance between the object and the light source: the
unit for the distance $\kappa$ is the scene width, the ratio
$\sigma_3/\sigma_4$ represents ``closeness to rank 3'', the errors are relative
in the Frobenius norm.}
\centering
\begin{tabular}{cccc}
\hline
$\kappa$ & $\sigma_3/\sigma_4$ & $E_{\text{lights}}$ & $E_{\text{surface}}$ \\
\hline
$\infty$ & $1.85\cdot 10^{15\strut}$ & $1.00\cdot 10^{-15}$ 
	& $2.69\cdot 10^{-4}$ \\
    1000 & $2.19\cdot 10^{3}$ & $1.95\cdot 10^{-4}$ & $1.39\cdot 10^{-3}$ \\
     100 & $2.18\cdot 10^{2}$ & $1.95\cdot 10^{-3}$ & $1.41\cdot 10^{-2}$ \\
      10 & $2.15\cdot 10^{1}$ & $1.95\cdot 10^{-2}$ & $1.45\cdot 10^{-1}$ \\
       1 & $1.77$ & $4.52\cdot 10^{-1}$ & $3.89$ \\
\hline
\end{tabular}
\end{table}

It is immediate to observe that when $\kappa$ takes value close to 1, 
the error produced in the light directions is amplified in
the object reconstruction, leading to unacceptable results.
We observed that the algorithm may fail in some situation, as the
deviation from ideality can lead to a non positive definite matrix $G$ in
\eqref{eq:diag}, causing an unrecoverable error.
This is one of the drawbacks of the numerical method, that must be faced in
future research.

\subsection{Experimental data sets}

An experimental data set partially satisfying the assumptions required by the
reconstruction method was generated as follows.
A seashell (approximate width 10~cm) was placed face up on a horizontal desk
with a black background.
The camera was held by a tripod about 1~m above the seashell.
The flat background was intended to reproduce homogeneous
Dirichlet boundary conditions for the observed surface.
The desk, standing in the open air under direct sunlight, was rotated in
order to take 20 pictures of the seashell with different lighting directions,
according to the shooting procedure described in Section~\ref{sec:orient}. The
resulting data set is displayed in \figurename~\ref{shelldata}.

\begin{figure}[hbt]
\begin{center}
\includegraphics[width=.8\columnwidth]{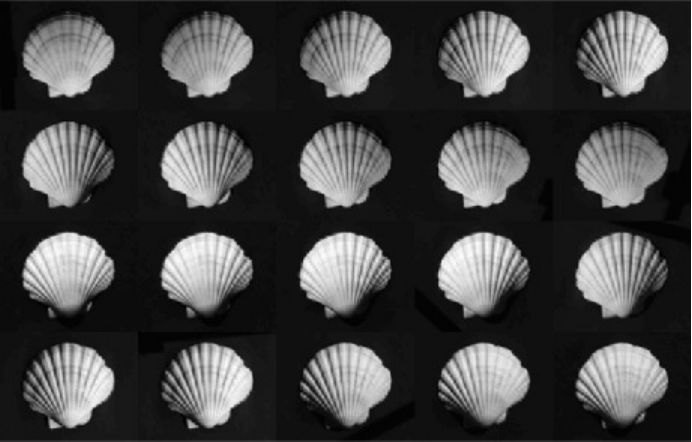}
\caption{The \textsc{shell} data set, consisting in 20 images
corresponding to different sunlight directions.}
\label{shelldata}
\end{center}
\end{figure}

\begin{figure}[hbt]
\begin{center}
\includegraphics[width=.8\columnwidth]{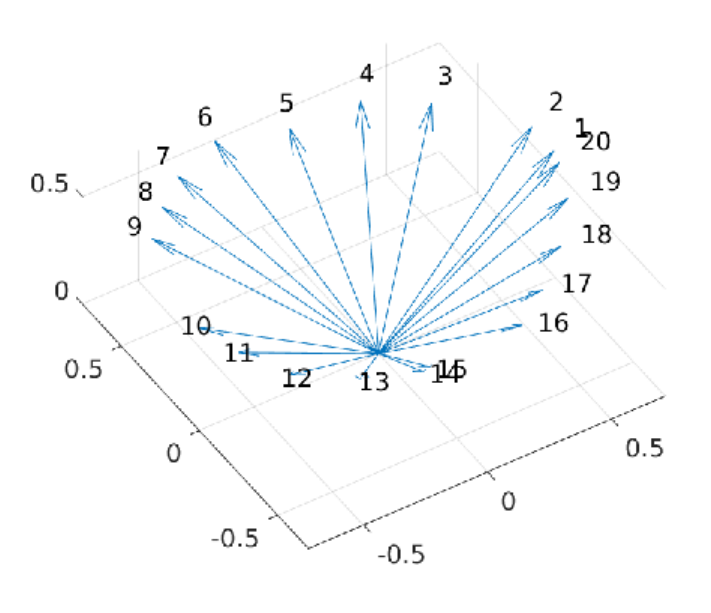}
\caption{Light directions identified by the reconstruction algorithm for
the \textsc{shell} data set.}
\label{shellights}
\end{center}
\end{figure}

While the relatively small distance between the camera and the object produces
images which cannot be represented through the orthographic projection model,
the sunlight rays can be assumed to be parallel.
So the lighting verifies the assumption of Lambert's model and we expect the
data matrix $M$ to be approximately rank-3.

\begin{figure}[hbt]
\begin{center}
\hfill
\includegraphics[width=.52\columnwidth]{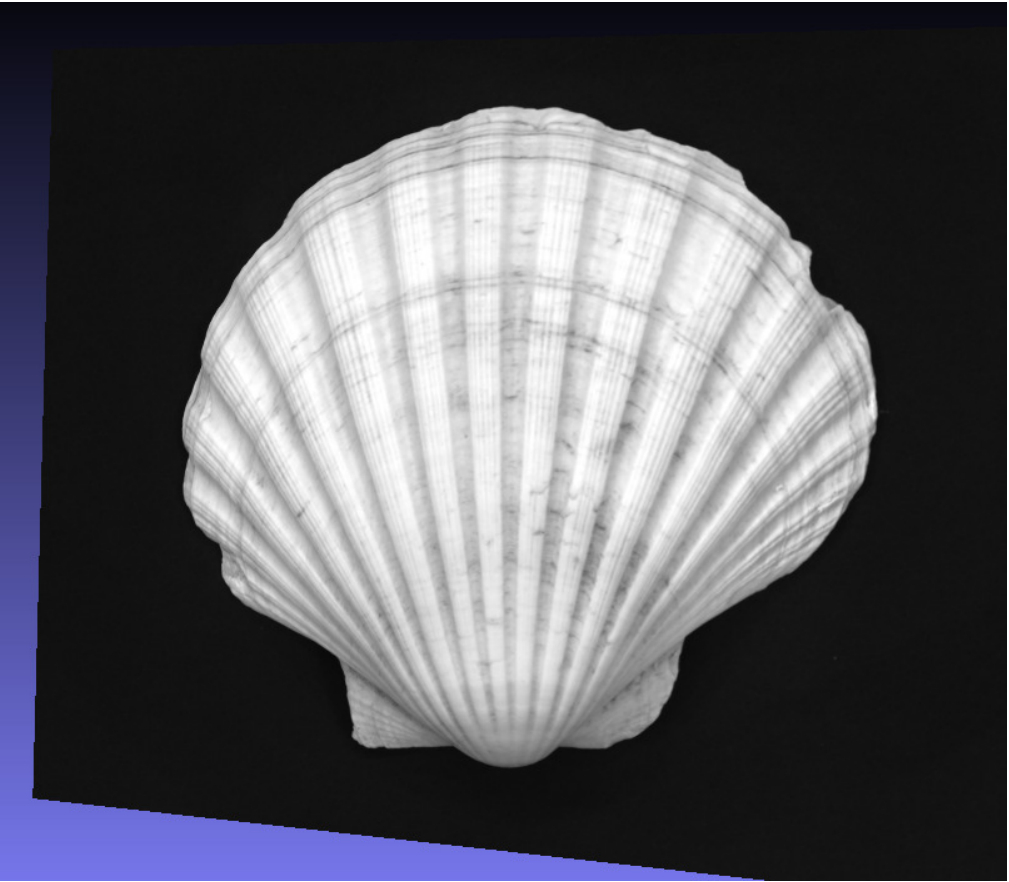}
\includegraphics[width=.45\columnwidth]{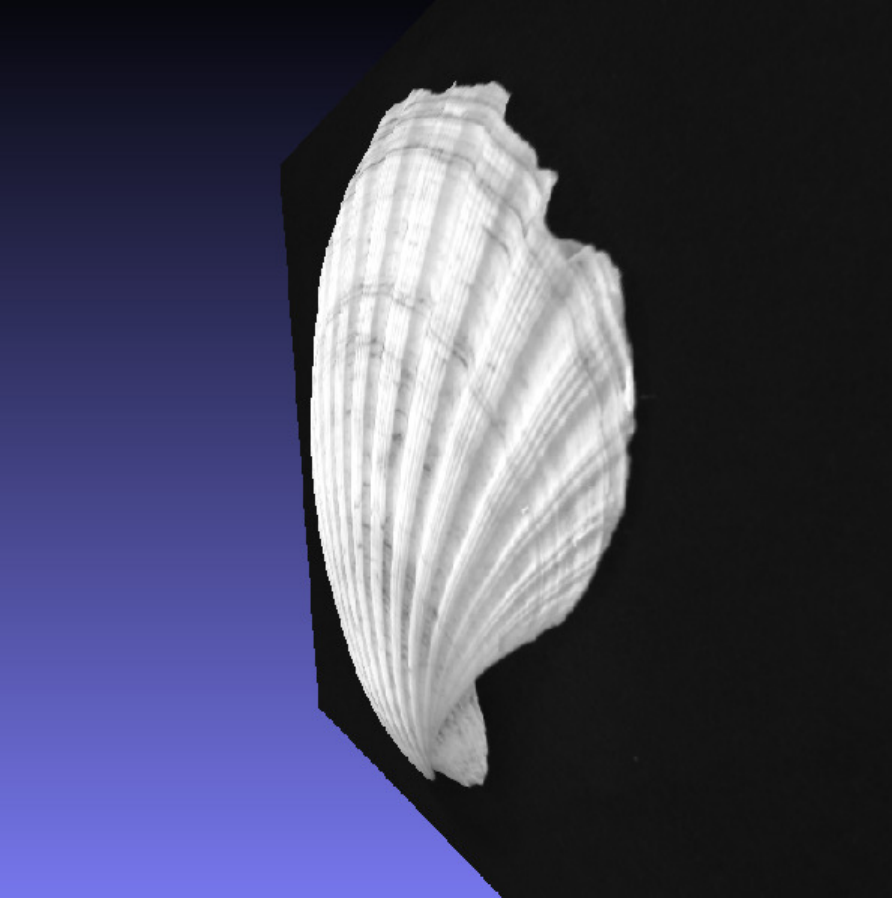}
\caption{Two different views of the 3D surface reconstructed from the
\textsc{shell} data set.}
\label{shellresults}
\end{center}
\end{figure}

The digital pictures were recorded in \emph{raw} mode at the resolution of
$3477\times 5220$ pixels, and for this particular numerical simulation they
were scaled to $885\times 705$ pixels.
The procedure described in Sections~\ref{sec:unknown} and~\ref{sec:orient}
identified 20 light vectors, displayed in \figurename~\ref{shellights}, that are
compatible with the sunlight position during the shooting process.

By solving the Poisson equation \eqref{poisson} with homogeneous Dirichlet
boundary conditions, one obtains the 3D model illustrated in
\figurename~\ref{shellresults} by two different views.
Compared to the original, the model appears slightly deformed by the deviation
from the orthographic projection model, but the reconstruction is quite
accurate and the computing time negligible, 2.8 seconds on an Intel Core i7
computer.
The deformation induced by the camera system could be corrected by camera
calibration techniques.

\begin{figure}[hbt]
\begin{center}
\includegraphics[width=.37\columnwidth]{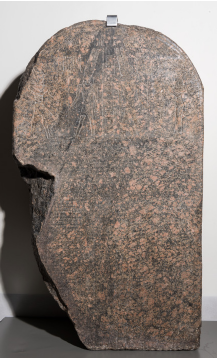}
\hfill
\includegraphics[width=.61\columnwidth]{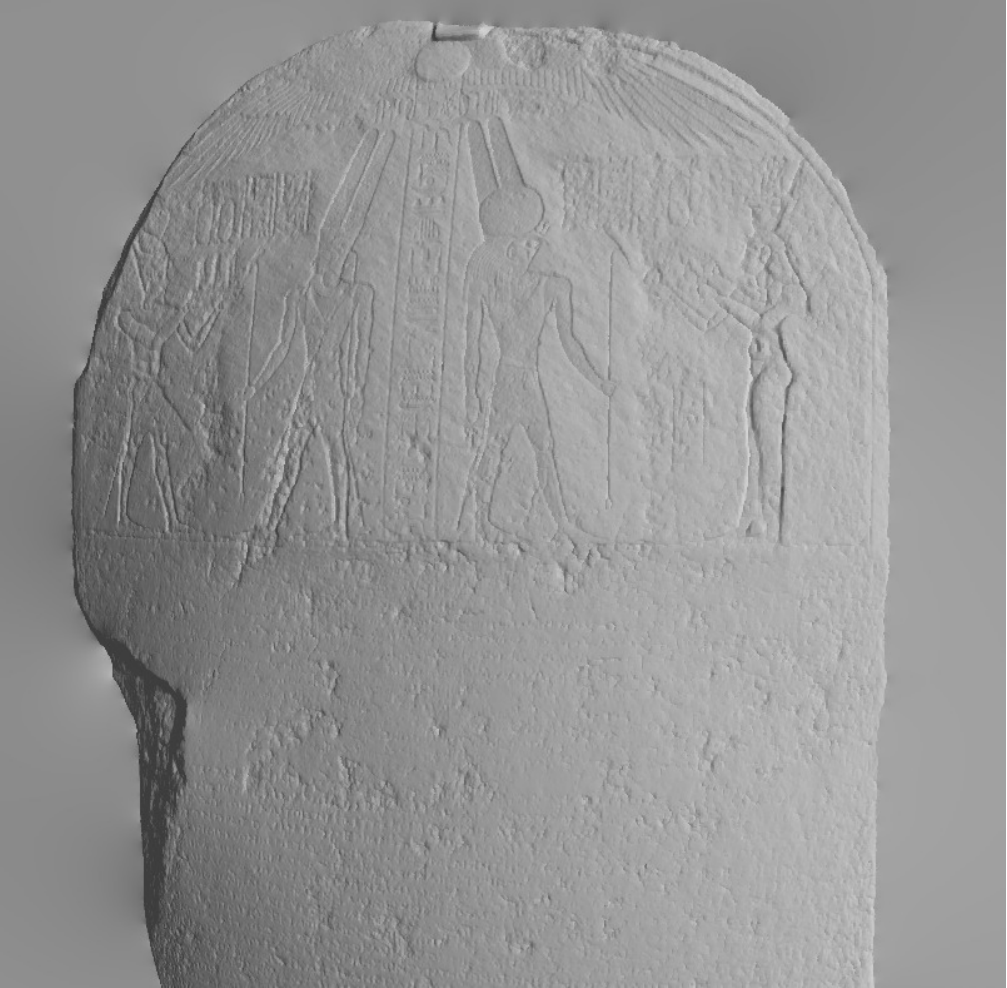}
\hfill
\caption{On the left, stela in honor of the general Callimachos, mentioning
Cleopatra and Caesarion; granite, Ptolemaic period, reign of Cleopatra VII, 39
BC. Thebes, Temple of Karnak (Courtesy of Museo Egizio, Torino, Italy). On the
right, 3D reconstruction obtained by the algorithm described in the paper.}
\label{steledata}
\end{center}
\end{figure}

As a second experiment based on a real data set, we processed the images of a
stela of the Ptolemaic period, exhibited at the Museo Egizio in Torino, Italy
(\url{https://www.museoegizio.it}); see \figurename~\ref{steledata}.
We thank the Museo Egizio, in particular Christian Greco, director of the
museum, and Marco Rossani, collection manager, for providing us the data set,
which is composed by 8 images.
A black mask was added around the stela in each image, in order to reproduce
Dirichlet boundary conditions.
The resolution of the \emph{raw} digital pictures is $7360\times 4912$ pixels.
The images were scaled to $1474\times 2208$ pixels to produce the result
depicted on the right of \figurename~\ref{steledata}, where the surface is displayed
after removing the albedo.
The computation took 16 seconds.

\begin{figure}[hbt]
\begin{center}
\includegraphics[width=.32\columnwidth]{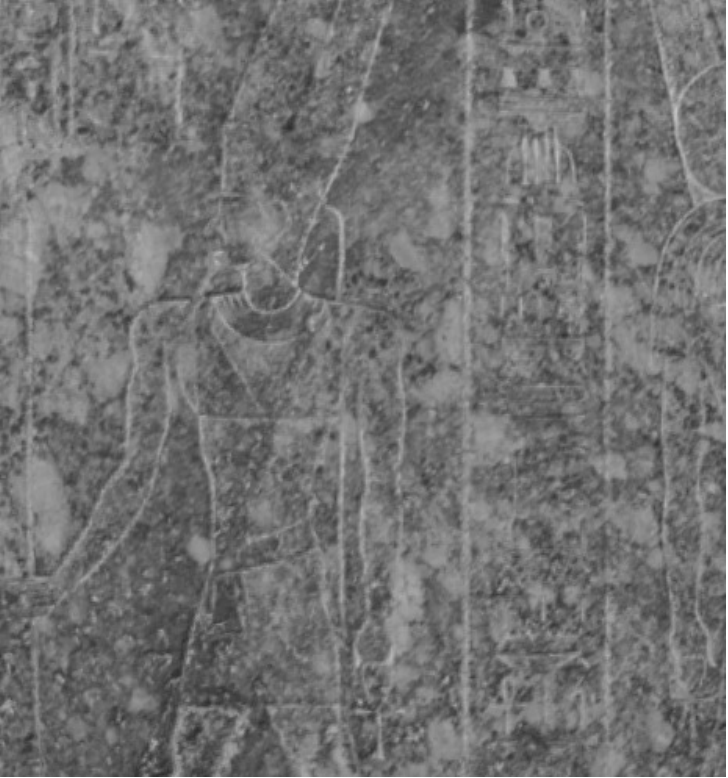}
\includegraphics[width=.32\columnwidth]{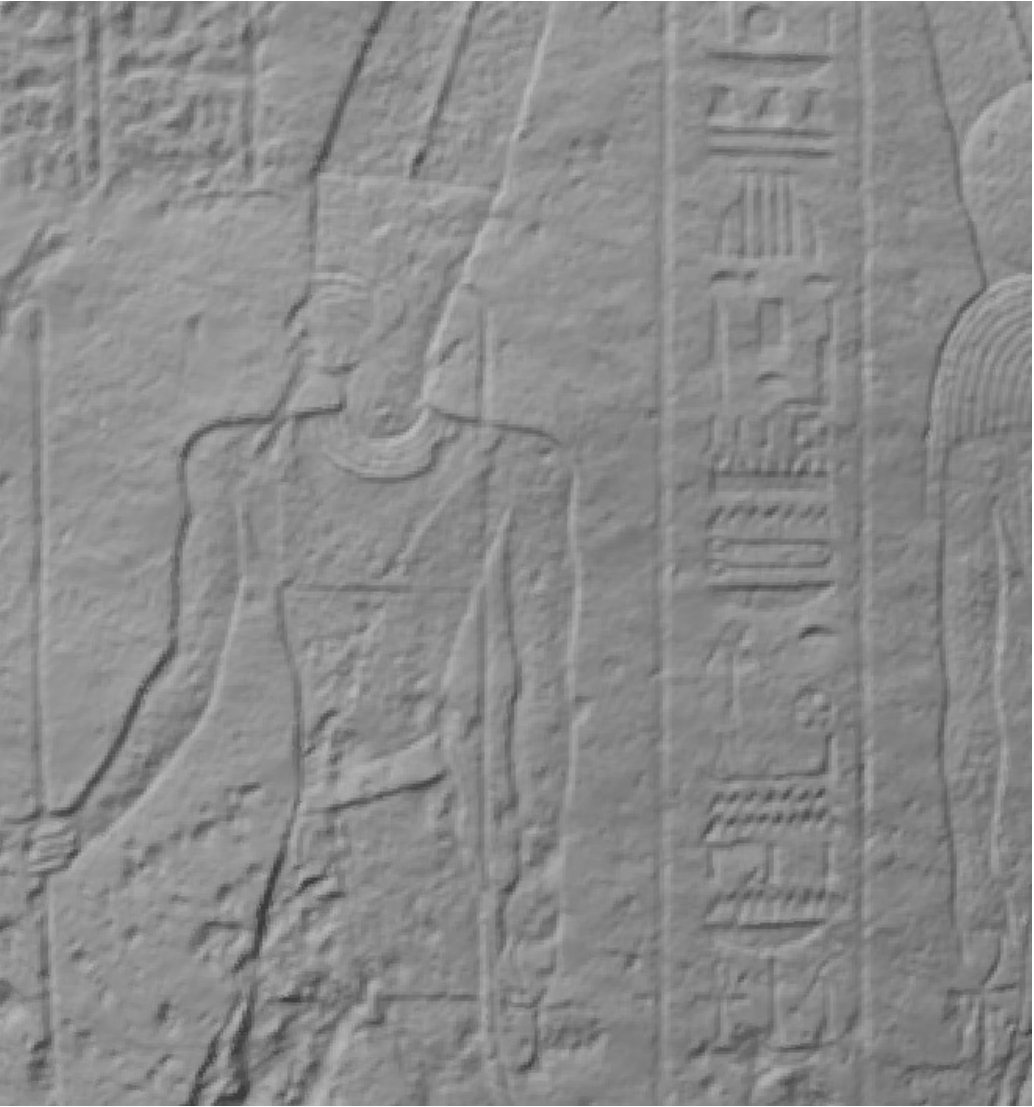}
\includegraphics[width=.32\columnwidth]{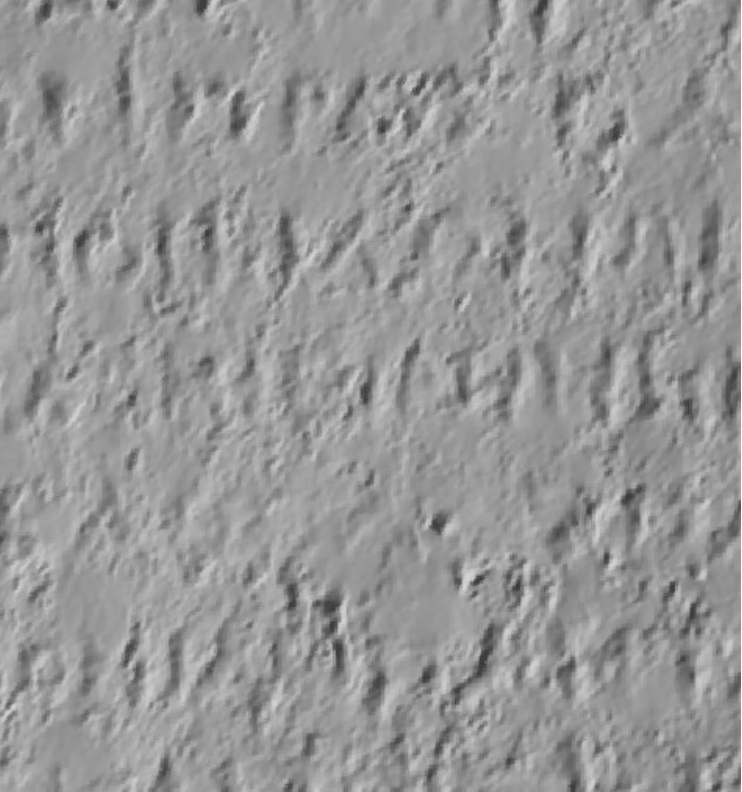}
\caption{Three details of the Ptolemaic stela reconstruction: the first two
show the same area with and without albedo, the third one a part containing
a writing.}
\label{stelepart}
\end{center}
\end{figure}

The pictures were taken in the exposition room of the museum using an
electronic flash as a light source, so they do not satisfy the assumptions of
the model, being both the camera and the light source at a quite small distance
from the object. Indeed, while the bas-relief details are accurately
reproduced, the reconstructed surface is spherically warped, compared to the
flatness of the real stela.

Some details of the reconstructed surface are displayed in
\figurename~\ref{stelepart}.
The left and central pictures show how albedo removal can lead to a cleaner
visualization of an engraving.
The image on the right is a part of the reconstruction obtained by processing
a $800\times 800$ sub-image of the original high resolution pictures. It shows
a writing, located in the central part of the stela.
To obtain a neat representation of small writings one would need high
resolution images only of this part, but this would introduce difficulties in
assigning boundary conditions to the Poisson equation.
We will face this problem in future work.

\section{Conclusions and future perspectives}\label{sec:future}

In this paper, under the ideal assumptions upon which Lambert's model is based,
we investigate the performance of the Hayakawa \cite{hayakawa1994} method
for estimating the lights position in photometric stereo.
This procedure can be applied when at least 6 images of an object are
available, each one with a different light source.
We show that this approach can be effectively used to reconstruct the 3D model
of an observed object, and we propose a procedure to avoid the indetermination
in the direction of the normal vectors, typical of photometric stereo.
We also make available the software we developed.

The accuracy of our method is investigated through numerical experiments on
both synthetic and real data sets.
Our experiments show that the algorithm is accurate under ideal conditions, and
that the lack from ideality produces, as expected, a spherical deformation on
the reconstructed surfaces.

Our future research work will be devoted to improving the method, in order to
make it applicable to real shooting conditions, in particular when the light
sources cannot be positioned sufficiently far away from the observed object. 
Moreover, real time processing of high resolution pictures requires a reduction
in the computing time, especially for what concerns the data matrix
factorization and the solution of the final large linear system.
Another aspect that needs further work is the treatment of the boundary
conditions required for the solution of equation \eqref{poisson}, whose
knowledge is not available in many applicative situations.

\section*{Acknowledgment}
We thank the Museo Egizio (Torino, Italy) for providing us the data set used in
the numerical experiments.
The research in this paper was partially supported by
the Fondazione di Sardegna 2017 research project ``Algorithms for Approximation
with Applications [Acube]'',
the INdAM-GNCS research project ``Tecniche numeriche per l'analisi delle reti
complesse e lo studio dei problemi inversi'', 
and the Regione Autonoma della Sardegna research project ``Algorithms and
Models for Imaging Science [AMIS]'' (RASSR57257, intervento finanziato con
risorse FSC 2014-2020 - Patto per lo Sviluppo della Regione Sardegna).
AC and CF gratefully acknowledge Regione Autonoma della
Sardegna for the financial support provided under the Operational Programme
P.O.R. Sardegna F.S.E. (European Social Fund 2014-2020 - Axis III Education
and Formation, Objective 10.5, Line of Activity 10.5.12).


\end{document}